\documentclass[a4paper,11pt]{amsart}
\usepackage{amsmath,amssymb,amsfonts,amsthm}

\usepackage[alphabetic]{amsrefs}
\usepackage{bm}
\usepackage{graphicx}
\usepackage{ascmac}
\usepackage[all]{xy}
\usepackage{amsthm}
\usepackage[top=25mm,bottom=25mm,left=30mm,right=30mm]{geometry}
\usepackage{tikz}
\usetikzlibrary{patterns}
\usepackage{pxpgfmark}
\usepackage{multirow}
\usetikzlibrary{intersections,arrows,calc,decorations.markings}
\newtheorem{theorem}{Theorem}[section]
\newtheorem{proposition}[theorem]{Proposition}
\newtheorem{conjecture}[theorem]{Conjecture}
\newtheorem{corollary}[theorem]{Corollary}
\newtheorem{lemma}[theorem]{Lemma}

\theoremstyle{definition}
\newtheorem{remark}[theorem]{Remark}
\newtheorem{example}[theorem]{Example}
\newtheorem{definition}[theorem]{Definition}

\makeatletter
 
 \@addtoreset{equation}{section}
\makeatother
\def\bb{\mathbf{b}}
\def\cc{\mathbf{c}}
\def\dd{\mathbf{d}}
\def\ee{\mathbf{e}}
\def\ff{\mathbf{f}}
\def\gg{\mathbf{g}}

\def\xx{\mathbf{x}}
\def\yy{\mathbf{y}}
\def\TT{\mathbb{T}}

\def\PP{\mathbb{P}}
\def\ZZ{\mathbb{Z}}

\def\Acal{\mathcal{A}}
\def\Fcal{\mathcal{F}}
\def\Xcal{\mathcal{X}}

\def\QQsf{\mathbb{Q}_{\text{sf}}}

\title[$f$-vectors and $d$-vectors in cluster algebras of finite type or rank $2$]{Relation between $f$-vectors and $d$-vectors \\in cluster algebras of finite type or rank 2}

\author{Yasuaki Gyoda}
\address{Graduate School of Mathematics, Nagoya University, Chikusa-ku, Nagoya, 464-8602 Japan}
\email{m17009g@math.nagoya-u.ac.jp}
\keywords{cluster algebra, mutation, f-vector, d-vector}
\subjclass[2010]{13F60}
\begin{document}
\maketitle
\begin{abstract} 
We study $f$-vectors, which are the maximal degree vectors of $F$-polynomials in cluster algebra theory. For a cluster algebra of finite type, we find that positive $f$-vectors correspond with $d$-vectors, which are exponent vectors of denominators of cluster variables. Furthermore, using this correspondence and properties of $d$-vectors, we prove that cluster variables in a cluster are uniquely determined by their $f$-vectors when the cluster algebra is of finite type or rank $2$.
\end{abstract} 
\section{Introduction and main theorems}
\emph{Cluster algebras} are commutative subalgebras of the rational function fields. They are generated by \emph{cluster variables} in \emph{clusters}, and cluster variables are obtained by applying \emph{mutations} repeatedly starting from the initial cluster. They are defined by \cite{fzi} to study the canonical basis or total positivity. Today, we know that they are related to many mathematical subjects. For example, by regarding a mutation as quiver transformation or triangulation of a marked surface, a structure of cluster algebras appears in representation theory of quivers \cites{bmrrt,birs} or higher Teichm\"uller theory \cites{fst,fg09}. Also, by considering a mutation of the Markov quiver, a new combinatorial approach to solve the Unicity Conjecture about Markov numbers was given in number theory \cites{cs,ros}.

Cluster algebras of \emph{finite type} and \emph{rank $2$} are important classes in cluster algebra theory.
Cluster algebras of finite type have finitely many cluster variables. They are introduced by \cite{fzi} and are classified completely by \cite{fzii}. They have connections with Dynkin diagrams or (real) root systems in Lie algebras, and they are applied to the logarithm identities, $T,Y$-systems and so on \cites{n,iikns,iikkn1,iikkn2}. 

Cluster algebras of rank 2 have two cluster variables in every cluster. Since they have the simplest structures in cluster algebras with infinitely many cluster variables, they are studied to understand other classes \cites{pz,ls,llz}.

The main topic of this paper is a relation between \emph{$f$-vectors} and \emph{$d$-vectors} in cluster algebras of finite type. Here, $f$-vectors are introduced in \cites{fk} as the maximal degree vectors of $F$-polynomials, and $F$-polynomials was introduced in \cite{fziv}. On the other hand, $d$-vectors are exponent vectors of monomials of denominators of cluster variables. They are introduced by \cites{fzii,fziv}. Though definitions of these two vectors are independent of each other, previous works \cites{fziv, rs,fg} suggested that they have some similar properties in cluster algebras of finite type or rank 2. In this paper, we give the following simple relation between $f$-vectors and $d$-vectors (Theorem \ref{main}):
\begin{align*}
\ff_{i;t}=[\dd_{i;t}]_+.
\end{align*}
This relation means that they are the same vectors in almost every situation. By this identification, we can study properties of $f$-vectors, which are not well-known yet, by properties of $d$-vectors. In this paper, we give a partial solution of the Uniqueness Conjecture \cite{gy}*{Conjecture 4.4}, that is, cluster variables in a cluster are uniquely determined by their $f$-vectors in cluster algebras of finite type or rank 2.

The organization of the paper is as follows: in the rest of this section, we define mutations, cluster algebras, $d$-vectors and $f$-vectors. After that, we describe the main theorem (Theorem \ref{main}) in the paper, that is, a simple relation between $d$-vectors and $f$-vectors in cluster algebras of finite type. Furthermore, we describe an application of the main theorem to the Uniqueness Conjecture (Theorem \ref{sub}). In Section 2, we give a proof of Theorem \ref{main}. In Section 3, we give a proof of Theorem \ref{sub} (1) by using Theorem \ref{main} and some properties of $d$-vectors. In Section 4, we give a proof of Theorem \ref{sub} (2) by using a description of entries of $d$-vectors. In Section 5, we generalize the cluster expansion formula given by \cite{llz} to the principal coefficients version along \cite{ls}, and we give the restoration formula of the $F$-polynomials from the $f$-vectors.

\subsection{Seed mutations and cluster algebras}
We start by recalling definitions of seed mutations and cluster patterns according to \cite{fziv}. A \emph{semifield} $\mathbb P$ is an abelian multiplicative group equipped with an addition $\oplus$ which is distributive over the multiplication. We particularly make use of the following two semifields.

Let $\mathbb Q_{\text{sf}}(u_1,\dots,u_{\ell})$ be the set of rational functions in $u_1,\dots,u_{\ell}$ which have subtraction-free expressions. Then, $\mathbb Q_{\text{sf}}(u_1,\dots,u_{\ell})$ is a semifield by the usual multiplication and addition. We call it the \emph{universal semifield} of $u_1,\dots,u_{\ell}$ (\cite{fziv}*{Definition 2.1}).

Let Trop$(u_1,\dots, u_\ell)$ be the abelian multiplicative group freely generated by the elements $u_1,\dots,u_\ell$. Then, $\text{Trop}(u_1,\dots,u_{\ell})$ is a semifield by the following addition: 
\begin{align}
\prod_{j=1}^\ell u_j^{a_j} \oplus \prod_{j=1}^{\ell} u_j^{b_j}=\prod_{j=1}^{\ell} u_j^{\min(a_j,b_j)}.
\end{align}
We call it the \emph{tropical semifield} of $u_1,\dots,u_\ell$ (\cite{fziv}*{Definition 2.2}).
For any semifield $\PP$ and $p_1, \dots, p_{\ell}\in\PP$, there exists a unique semifield homomorphism $\pi$ such that
\begin{align} \label{qsfuniv}
	\pi:\QQsf(y_1, \dots, y_{\ell}) &\longrightarrow \PP\\
	y_i &\longmapsto p_i. \nonumber
\end{align} 
For $F(y_1,\dots,y_\ell ) \in \QQsf(y_1, \dots, y_{\ell})$, we denote 
\begin{align}
	F|_{\PP}(p_1, \dots, p_{\ell}):=\pi(F(y_1, \dots, y_\ell)).
\end{align}
and we call it the \emph{evaluation} of $F$ at $p_1, \dots, p_{\ell}$.
We fix a positive integer $n$ and a semifield $\PP$. Let $\mathbb{ZP}$ be the group ring of $\mathbb{P}$ as a multiplicative group. Since $\mathbb{ZP}$ is a domain (\cite{fzi}*{Section 5}), its total quotient ring is a field $\mathbb{Q}(\mathbb P)$. Let $\mathcal{F}$ be the field of the rational functions in $n$ indeterminates with coefficients in $\mathbb{Q}(\mathbb P)$. 

A \emph{labeled seed with coefficients in $\PP$} is a triplet $(\mathbf{x}, \mathbf{y}, B)$, where
\begin{itemize}
\item $\mathbf{x}=(x_1, \dots, x_n)$ is an $n$-tuple of elements of $\mathcal F$ forming a free generating set of $\mathcal F$.
\item $\mathbf{y}=(y_1, \dots, y_n)$ is an $n$-tuple of elements of $\mathbb{P}$.
\item $B=(b_{ij})$ is an $n \times n$ integer matrix which is \emph{skew-symmetrizable}, that is, there exists a positive integer diagonal matrix $S$ such that $SB$ is skew-symmetric. Also, we call $S$ a \emph{skew-symmetrizer} of $B$.
\end{itemize}

We say that $\xx$ is a \emph{cluster} and refer to $x_i,y_i$ and $B$ as the \emph{cluster variables}, the \emph{coefficients} and the \emph{exchange matrix}, respectively.

Throughout the paper, for an integer $b$, we use the notation $[b]_+=\max(b,0)$. We note that
\begin{align}\label{eq:b--b}
b=[b]_+-[-b]_+.
\end{align}
Let $(\mathbf{x}, \mathbf{y}, B)$ be a labeled seed with coefficients in $\PP$, and let $k \in\{1,\dots, n\}$. The \emph{seed mutation $\mu_k$ in direction $k$} transforms $(\mathbf{x}, \mathbf{y}, B)$ into another labeled seed $\mu_k(\mathbf{x}, \mathbf{y}, B)=(\mathbf{x'}, \mathbf{y'}, B')$ defined as follows:
\begin{itemize}
\item The entries of $B'=(b'_{ij})$ are given by 
\begin{align} \label{eq:matrix-mutation}
b'_{ij}=\begin{cases}-b_{ij} &\text{if $i=k$ or $j=k$,} \\ 
b_{ij}+\left[ b_{ik}\right] _{+}b_{kj}+b_{ik}\left[ -b_{kj}\right]_+ &\text{otherwise.}
\end{cases}
\end{align}
\item The coefficients $\mathbf{y'}=(y'_1, \dots, y'_n)$ are given by 
\begin{align}\label{eq:y-mutation}
y'_j=
\begin{cases}
y_{k}^{-1} &\text{if $j=k$,} \\ 
y_j y_k^{[b_{kj}]_+}(y_k \oplus 1)^{-b_{kj}} &\text{otherwise.}
\end{cases}
\end{align}
\item The cluster variables $\mathbf{x'}=(x'_1, \dots, x'_n)$ are given by
\begin{align}\label{eq:x-mutation}
x'_j=\begin{cases}\dfrac{y_k\mathop{\prod}\limits_{i=1}^{n} x_i^{[b_{ik}]_+}+\mathop{\prod}\limits_{i=1}^{n} x_i^{[-b_{ik}]_+}}{(y_k\oplus 1)x_k} &\text{if $j=k$,}\\
x_j &\text{otherwise.}
\end{cases}
\end{align}
\end{itemize}

Let $\mathbb{T}_n$ be the \emph{$n$-regular tree} whose edges are labeled by the numbers $1, \dots, n$ such that the $n$ edges emanating from each vertex have different labels. We write 
\begin{xy}(0,1)*+{t}="A",(10,1)*+{t'}="B",\ar@{-}^k"A";"B" \end{xy} 
to indicate that vertices $t,t'\in \mathbb{T}_n$ are joined by an edge labeled by $k$. We fix a vertex $t_0\in \TT_n$, which is called the \emph{rooted vertex}.

A \emph{cluster pattern with coefficients in $\PP$} is an assignment of every labeled seed $\Sigma_t=(\mathbf{x}_t, \mathbf{y}_t,B_t)$ with coefficients in $\PP$ to every vertex $t\in \mathbb{T}_n$ such that the labeled seeds $\Sigma_t$ and $\Sigma_{t'}$ assigned to the endpoints of any edge 
\begin{xy}(0,1)*+{t}="A",(10,1)*+{t'}="B",\ar@{-}^k"A";"B" \end{xy} 
are obtained from each other by a seed mutation in direction $k$. Elements of $\Sigma_t$ are denoted as follows:
\begin{align} \label{den:seed_at_t}
\mathbf{x}_t=(x_{1;t},\dots,x_{n;t}),\ \mathbf{y}_t=(y_{1;t},\dots,y_{n;t}),\ B_t=(b_{ij;t}).
\end{align}
In particular, at $t_0$, we denote
\begin{align} \label{initialseed}
\mathbf{x}=\mathbf{x}_{t_0}=(x_1,\dots,x_n),\ \mathbf{y}=\mathbf{y}_{t_0}=(y_1,\dots,y_n),\ B=B_{t_0}=(b_{ij}).
\end{align}

When we want to emphasize that the initial matrix is $B$, we denote by $\Sigma^B_t$ a labeled seed associated with a vertex $t$. 
For seeds $\Sigma_t$ and $\Sigma_s$ in a cluster pattern, if there exists mutation sequence $\mu$ such that $\Sigma_s=\mu(\Sigma_t)$, then we say that $\Sigma_t$ is \emph{mutation equivalent} to $\Sigma_s$.
\begin{definition}
A \emph{cluster algebra} $\Acal$ associated with a cluster pattern $v\mapsto \Sigma_v$ is the $\ZZ\PP$-subalgebra of $\Fcal$ generated by $\{x_{i;t}\}_{1\leq i\leq n, t\in \TT_n}$.
 \end{definition}

The degree $n$ of the regular tree $\TT_n$ is called the \emph{rank} of $\Acal$, and $\Fcal$ is the \emph{ambient field} of $\Acal$. 

We also denote by $\mathcal{A}(\xx,\yy,B)$ a cluster algebra with the initial seed $(\xx,\yy,B)$. 
\begin{example}\label{A2}
We give an example of cluster algebras. This example is based on \cite{fziv}*{Example 2.10} (but it is different from \cite{fziv} with respect to the way of labeling edges). 
Let $n=2$, and we consider a tree $\TT_2$ whose edges are labeled as follows:
\begin{align}\label{A2tree}
\begin{xy}
(-10,0)*+{\dots}="a",(0,0)*+{t_0}="A",(10,0)*+{t_1}="B",(20,0)*+{t_2}="C", (30,0)*+{t_3}="D",(40,0)*+{t_4}="E",(50,0)*+{t_5}="F", (60,0)*+{\dots}="f"
\ar@{-}^{2}"a";"A"
\ar@{-}^{1}"A";"B"
\ar@{-}^{2}"B";"C"
\ar@{-}^{1}"C";"D" 
\ar@{-}^{2}"D";"E"
\ar@{-}^{1}"E";"F" 
\ar@{-}^{2}"F";"f" 
\end{xy}.
\end{align}
Let $B=\begin{bmatrix}
 0 & 1 \\
 -1 & 0
\end{bmatrix}
$ be the initial exchange matrix at $t_0$.
Then coefficients and cluster variables are given by Table 1.
\begin{table}[ht]
\caption{Coefficients and cluster variables in type~$A_2$}
$
\begin{array}{|c|cc|cc|}
\hline
&&&&\\[-4mm]
t& \hspace{25mm}\yy_t &&& \xx_t \hspace{30mm}\\
\hline
&&&&\\[-3mm]
0 &y_1 & y_2& x_1& x_2 \\[1mm]
\hline
&&&&\\[-3mm]
1& \dfrac{1}{y_1} & \dfrac{y_1y_2}{y_1\oplus 1} & \dfrac{x_2+y_1}{(y_1\oplus 1)x_1}& x_2 \\[3mm]
\hline
&&&&\\[-3mm]
2& \dfrac{y_2}{y_1y_2\oplus y_1\oplus 1} & \dfrac{y_1\oplus 1}{y_1y_2} & \dfrac{x_2+y_1}{(y_1\oplus 1)x_1} & \dfrac{x_1y_1y_2+y_1+x_2}{(y_1y_2\oplus y_1\oplus 1)x_1x_2} \\[3mm]
\hline
&&&&\\[-3mm]
3& \dfrac{y_1y_2\oplus y_1\oplus1}{y_2} & \dfrac{1}{y_1(y_2\oplus 1)} & \dfrac{x_1y_2+1}{(y_2\oplus 1)x_2} & \dfrac{x_1y_1y_2+y_1+x_2}{(y_1y_2\oplus y_1\oplus 1)x_1x_2} \\[3mm]
\hline
&&&&\\[-2mm]
4& \dfrac{1}{y_2} &y_1(y_2\oplus 1) & \dfrac{x_1y_2+1}{(y_2\oplus 1)x_2} & x_1 \\[3mm]
\hline
&&&&\\[-2mm]
5& y_2 & y_1 & x_2 & x_1\\[1mm]
\hline
\end{array}
$
\end{table}

Therefore, we have
\begin{align*}
\Acal(\xx,\yy,B)=\ZZ\PP\left[x_1,x_2,\dfrac{x_2+y_1}{(y_1\oplus 1)x_1},\dfrac{x_1y_1y_2+y_1+x_2}{(y_1y_2\oplus y_1\oplus 1)x_1x_2}, \dfrac{x_1y_2+1}{(y_2\oplus 1)x_2}\right].
\end{align*}
\end{example}

Next, in order to define classes of cluster algebras which we deal with in this paper, we define non-labeled seeds according to \cite{fziv}. For a cluster pattern $v\mapsto \Sigma_v$, we introduce the following equivalence relations of labeled seeds: we say that \begin{align*}
\Sigma_t=(\xx_t, \yy_t, B_t),\quad \xx_t=(x_{1;t,}\dots,x_{n;t}),\quad
\yy_t=(y_{1;t},\dots,y_{n;t}),\quad B_t=(b_{ij;t})
\end{align*}
and
\begin{align*}
\Sigma_s=(\xx_{s}, \yy_{s}, B_{s}),\quad \xx_s=(x_{1;s},\dots,x_{n;s}),\quad
\yy_s=(y_{1;s},\dots,y_{n;s}),\quad B_s=(b_{ij;s})
\end{align*}
are \emph{equivalent} if there exists a
permutation~$\sigma$ of indices~$1, \dots, n$ such that
\begin{align*}
x_{i;s} = x_{\sigma(i);t}, \quad y_{j;s} = y_{\sigma(j);t}, \quad
b_{ij;s} = b_{\sigma(i), \sigma(j);t}
\end{align*}
for all~$i$ and~$j$.
We denote by $[\Sigma]$ the equivalent classes represented by a labeled seed
$\Sigma$ and call it \emph{non-labeled seed}. Also, We define a \emph{(non-labeled) clusters} $[\xx]$ as the set ignored the order of a labeled cluster. Abusing notation, we abbreviate $[\xx]$ to $\xx$. 

\begin{definition}
The \emph{exchange graph} of a cluster algebra is the regular connected graph whose vertices are non-labeled seeds of the cluster pattern and whose edges connect non-labeled seeds related by a single mutation.
\end{definition}
Using the exchange graph, we define cluster algebras of finite type.
\begin{definition}
A cluster algebra $\Acal$ is of \emph{finite type} if the exchange graph of $\Acal$ is a finite graph.
\end{definition}
We say that $B$ is \emph{bipartite} if there is a function $\varepsilon:\{1,\dots,n\} \rightarrow\{1,-1\}$ such that for all $i$ and $j$, 
\begin{align}\label{bipartitecondition}
b'_{ij}>0 \Rightarrow \begin{cases}
\varepsilon(i)=1,\\
\varepsilon(j)=-1.
\end{cases}
\end{align}.

For an exchange matrix $B$, we define $A(B)=(a_{ij})$ as 
\begin{align*}
a_{ij}=\begin{cases}
2\quad &\text{if } i=j;\\
-|b'_{ij}| & \text{if } i\neq j.
\end{cases}
\end{align*}
If $A(B)$ is a Cartan matrix, then we say that $B$ is of \emph{finite Cartan type}.

\begin{remark}\label{cartanexchange}
If $\mathcal{A}=\mathcal{A}(\xx,\yy,B)$ is of finite type, then the initial matrix $B$ is mutation equivalent to a bipartite matrix $B'$. Furthermore, by permuting their indices appropriately, we can choose $B'$ as one of finite Cartan type (see \cite{fzii}*{Theorem 1.8, Theorem 7.1}). If the initial matrix $B$ of $\Acal$ is mutation equivalent to $B'$ which is finite Cartan $X_n$ type, then there exists a bijection between almost positive roots of $X_n$ type and cluster variables of $\Acal$ (see \cite{fzii}*{Theorem 1.9}). 
\end{remark}

\subsection{$d$-vectors and $f$-vectors} 
In this subsection, we define $d$-vectors and $f$-vectors. First, we define $d$-vectors according to \cites{fzii,fziv}. 

Let $\Acal$ be a cluster algebra.
By the \emph{Laurent phenomenon} \cite{fziv}*{Theorem 3.5}, every cluster variable
$x_{i;t} \in \Acal$ can be uniquely written as
\begin{align}\label{eq:Laurent-normal-form}
x_{i;t} = \frac{N_{i;t}(x_1, \dots, x_n)}{x_1^{d_{1i;t}} \cdots x_n^{d_{ni;t}}},\quad d_{ki;t}\in \ZZ,
\end{align}
where $N_{i;t}(x_1, \dots, x_n)$ is a polynomial with coefficients in~$\ZZ \PP$
which is not divisible by any initial cluster variable~$x_i\in\xx$.
We define the \emph{$d$-vector} $\dd_{j;t}$ as the degree vector of $x_{j;t}$, that is, 
\begin{align}
\dd_{i;t}^{B;t_0}=\dd_{i;t}=\begin{bmatrix}d_{1i;t}\\ \vdots \\ d_{ni;t} \end{bmatrix}.
\end{align}
We define a \emph{$D$-matrix} $D_t^{B;t_0}$ as 
\begin{align}
D_t^{B;t_0}:=(\dd_{1;t},\dots,\dd_{n;t}).
\end{align}
We remark that
$\dd_{i;t}$ is independent of the choice of $\PP$ (see \cite{fziv}*{Section 7}). Therefore, in a cluster algebra $\Acal$, if $x_{i;t}=x_{j;s}$, then we have $\dd_{i;t}=\dd_{j;s}$.
Moreover, $d$-vectors are also given by the following recursion: 
for any $i\in\{1,\dots,n\}$,
\begin{align*}
\dd_{i;t_0}=-\mathbf{e}_i,
\end{align*}
and for any \begin{xy}(0,1)*+{t}="A",(10,1)*+{t'}="B",\ar@{-}^k"A";"B" \end{xy}, 
\begin{align}\label{dvectorrecursion}
\mathbf{d}_{i;{t'}}&=\begin{cases}
\mathbf{d}_{i;t} \ \ & \text{if } i\neq k;\\
-\mathbf{d}_{k;t}+\max \left(\mathop{\sum}\limits_{j=1}^n[b_{jk;t}]_+\mathbf{d}_{j;t},\ +\mathop{\sum}\limits_{j=1}^n[-b_{jk;t}]_+\mathbf{d}_{j;t}\right)\ \ &\text{if } i=k,
\end{cases}
\end{align}
where $\ee_i$ is a standard basis element and the operation $\max$ on vectors are performed component-wise. By this way of definition, since $d$-vectors depend only on exchange matrices, we can regard $d$-vectors as vectors associated with vertices of $\TT_n$.

Next, we define $f$-vectors according to \cite{fg}. We will give some preparations.

We say that a cluster pattern $v\mapsto \Sigma_v$ (or a cluster algebra $\Acal$) has the \emph{principal coefficients} at the initial vertex $t_0$ if $\mathbb{P}=\text{Trop}(y_1,\dots,y_n)$ and $\mathbf{y}_{t_0}=(y_1,\dots,y_n)$. In this case, we denote $\Acal=\Acal_\bullet(B)$. For any $\Acal_\bullet(B)$ whose rank is $n$, any $t\in\TT_n$ and $i\in\{1,\dots,n\}$, we define the \emph{$F$-polynomial} $F^{B;t_0}_{i;t}(\yy)$ as 
\begin{align}
F^{B;t_0}_{i;t}(\yy)=x_{i;t}(x_1,\dots,x_n)|_{x_1=\cdots=x_n=1},
\end{align}
where $x_{i;t}(x_1,\dots,x_n)$ means the expression of $x_{i;t}$ by $x_1,\dots,x_n$. 

Using $F$-polynomials, we define $f$-vectors. Let $\Acal_{\bullet}(B)$ be a cluster algebra with principal coefficients at $t_0$. We denote by $f_{ij;t}$ the maximal degree of $y_i$ in $F_{j;t}^{B;t_0}(\yy)$. Then, we define the \emph{$f$-vector} $\ff_{i;t}$ as

\begin{align}
\ff_{i;t}^{B;t_0}=\ff_{i;t}=\begin{bmatrix}f_{1i;t}\\ \vdots \\ f_{ni;t} \end{bmatrix}.
\end{align}
We define the \emph{$F$-matrix} $F_t^{B;t_0}$ as 
\begin{align}
F_t^{B;t_0}:=(\ff_{1;t},\dots,\ff_{n;t}).
\end{align}
\begin{remark}\label{independentf}
For $\mathbf{b}=\begin{bmatrix}
b_1\\
\vdots\\
b_n
\end{bmatrix}$, we denote $[\mathbf{b}]_+=\begin{bmatrix}
[b_1]_+\\
\vdots\\
[b_n]_+
\end{bmatrix}$. By \cite{fg}*{Proposition 2.7}, $f$-vectors are the same as those defined by the following recursion: 
for any $i\in\{1,\dots,n\}$,
\begin{align*}
\ff_{i;t_0}=\mathbf{0},
\end{align*}
and for any \begin{xy}(0,1)*+{t}="A",(10,1)*+{t'}="B",\ar@{-}^k"A";"B" \end{xy}, 
\begin{align}\label{f-recursion}
\mathbf{f}_{i;{t'}}&=\begin{cases}
\mathbf{f}_{i;t} \ \ & \text{if } i\neq k;\\
-\mathbf{f}_{k;t}+\max \left([\mathbf{c}_{k;t}]_++\mathop{\sum}\limits_{j=1}^n[b_{jk;t}]_+\mathbf{f}_{j;t},\ [-\mathbf{c}_{k;t}]_++\mathop{\sum}\limits_{j=1}^n[-b_{jk;t}]_+\mathbf{f}_{j;t}\right)\ \ &\text{if } i=k,
\end{cases}
\end{align}
where $\cc_{i;t}$ is a \emph{$c$-vector}, which is defined by the following recursion:
For any $i\in\{1,\dots,n\}$
\begin{align*}
\cc_{i;t_0}=\ee_i,
\end{align*}
and for any \begin{xy}(0,1)*+{t}="A",(10,1)*+{t'}="B",\ar@{-}^k"A";"B" \end{xy}, 
\begin{align*}
\cc_{i;t'} =
\begin{cases}
-\cc_{i;t} & \text{if $i\neq k$;} \\[.05in]
\cc_{i;t} + [b_{ki;t}]_+ \ \cc_{k;t} +b_{ki;t} [-\cc_{k;t}]_+
 & \text{if $i=k$}.
 \end{cases}
\end{align*}

By this way of definition, since $f$-vectors depend only on exchange matrices, we can regard $f$-vectors as vectors associated with vertices of $\TT_n$. In this case, we remark that $f$-vectors are independent of the choice of coefficient system. So do $F$-matrices.
Furthermore, when we define $f$-vectors as these recursions, we have the following property: in a cluster algebra $\Acal$, if $x_{i;t}=x_{j;s}$, then we have $\ff_{i;t}=\ff_{j;s}$. It follows from \cite{cl2}*{Proposition 3 (i)} immediately.
\end{remark}
Since we know that $d$-vectors and $f$-vectors depend only on $B$ by above discussion, we abbreviate a cluster algebra $\Acal(\xx,\yy,B)$ to $\Acal(B)$ when we discuss properties of $d$-vectors or $f$-vectors.
\begin{example} Let $\Acal(B)$ be a cluster algebra given in Example \ref{A2}. Then $F$-polynomials, $F$-matrices, and $D$-matrices are given by Table \ref{A2FD}.
\begin{table}[ht]
\caption{$F$-polynomials, $F$- and $D$-matrices in type~$A_2$\label{A2FD}}
$
\begin{array}{|c|cc|c|c|}
\hline
&&&&\\[-4mm]
t& \hspace{0mm}F^{B; t_0}_{1;t}(\yy)&F^{B; t_0}_{2;t}(\yy) & F^{B; t_0}_t \hspace{0mm} & D^{B;t_0}_t \hspace{0mm}\\
\hline
&&&&\\[-3mm]
0& 1&1 
& \begin{bmatrix}
 0 & 0 \\
 0 & 0
\end{bmatrix}
&\begin{bmatrix}
 -1 & 0 \\
 0 & -1
\end{bmatrix}
 \\[4mm]
\hline
&&&&\\[-3mm]
1&y_1+1&1
&\begin{bmatrix}
 1 & 0 \\
 0 & 0
\end{bmatrix}
&\begin{bmatrix}
 1& 0 \\
 0 & -1
\end{bmatrix}\\[4mm]
\hline
&&&&\\[-3mm]
2& y_1+1&y_1y_2+y_1+1
&\begin{bmatrix}
 1 & 1 \\
 0 & 1
\end{bmatrix}
&\begin{bmatrix}
 1 & 1 \\
 0 & 1
\end{bmatrix}\\[4mm]
\hline
&&&&\\[-3mm]
3& y_2+1&y_1y_2+y_1+1
&\begin{bmatrix}
 0 & 1 \\
 1 & 1
\end{bmatrix}
&\begin{bmatrix}
 0 & 1 \\
 1 & 1
\end{bmatrix}\\[4mm]
\hline
&&&&\\[-3mm]
4& y_2+1&1
&\begin{bmatrix}
 0 & 0 \\
 1& 0
\end{bmatrix}
&\begin{bmatrix}
 0 & -1 \\
 1 & 0
\end{bmatrix} \\[4mm]
\hline
&&&&\\[-3mm]
5& 1&1
&\begin{bmatrix}
 0 & 0 \\
 0 & 0
\end{bmatrix}
&\begin{bmatrix}
 0 & -1 \\
 -1 & 0
\end{bmatrix}\\[4mm]
\hline
\end{array}
$
\end{table}
\end{example}

 We are ready to describe the main results in this paper.

\subsection{Main results}
The main result of this paper is the following theorem:
\begin{theorem}\label{main}
In a cluster algebra $\Acal(B)$ of finite type, for any $i\in\{1,\dots,n\}$ and $t\in\TT_n$, we have the following relation:
\begin{align}\label{fdeq}
\ff_{i;t}=[\dd_{i;t}]_+.
\end{align}
\end{theorem}

It is known that Theorem \ref{main} holds under the condition that the initial matrix $B$ is bipartite by combining Corollary 10.10 and Proposition 11.1 (1) in \cite{fziv}. When $B$ is a skew-symmetric matrix, Theorem \ref{main} has already proved by using 2-Carabi-Yau categories (see \cite{fk}*{Proposition 6.6}). We remove these conditions.
\begin{remark}\label{f=drank2}
In the case that $\Acal(B)$ is of rank 2, we have \eqref{fdeq} by combining Corollary 10.10 and Proposition 11.1 (1) in \cite{fziv}. If $\Acal$ is of neither finite type nor rank 2, Theorem \ref{main} does not hold generally. A counterexample is given by \cite{fk}*{Section 6.4} . 
\end{remark}

We give an application of Theorem \ref{main}. Let us introduce the \emph{Uniqueness Conjecture} in \cite{gy}:

\begin{conjecture}[\cite{gy}*{Conjecture 4.4}]\label{conjF}
In a cluster algebra $\Acal(B)$, for $t,s\in\TT_n$, $F_t^{B;t_0}=F_s^{B;t_0}$ implies that $\xx_t$ and $\xx_s$ are the same non-labeled cluster.
\end{conjecture}

This conjecture is also studied in viewpoint of representation theory of algebras. An $f$-vectors are a dimension vector of the corresponding indecomposable $\tau$-rigid module over an appropriate 2-Calabi-Yau tilted algebras in additive categorification by 2-Calabi-Yau categories. By using the correspondences, Conjecture \ref{conjF} is equivalent to the following problem: support $\tau$-tilting modules are uniquely determined by the set of dimension vectors of these indecomposable direct summands. This problem was solved in the case of skew-symmetric cluster algebras of finite type \cites{gp,r}, skew-symmetric cluster algebras of affine type \cite{fg17}, and cluster algebras of $C_n$ Dynkin type \cite{fgl}.

In the case that $\Acal$ is of (skew-symmetrizable) finite type or rank $2$, we prove Conjecture \ref{conjF} by showing the following statement:
\begin{theorem}\label{sub}
\begin{itemize}
\item[(1)]
In a cluster algebra of finite type, for any $t,s\in \TT_n$, if $(\ff_{1;t},\dots,\ff_{n;t})$ coincides with $(\ff_{1;s},\dots,\ff_{n;s})$ up to order, then $\xx_t$ and $\xx_s$ are the same non-labeled cluster.
\item[(2)]
In a cluster algebra of rank $2$, for $t,s\in\TT_2$, if $(\ff_{1;t},\ff_{2;t})$ coincides with $(\ff_{1;s},\ff_{2;s})$ up to order, then $\xx_t$ and $\xx_s$ are the same non-labeled cluster.
\end{itemize}
\end{theorem}
Theorem \ref{sub} is a theorem of slightly stronger form than Conjecture \ref{conjF}. In Conjecture \ref{conjF}, the order of the f-vectors is fixed, but in Theorem \ref{sub}, it is not.
\begin{remark}
In the case of cluster algebras of $A_n$ or $D_n$ type, Theorem \ref{sub} has already been proved by using marked surfaces \cite{gy}*{Corollary 4.8}.
\end{remark}
\subsection*{Acknowledgement}
The author would like to express his gratitude to Bernhard Keller for insightful comments about Theorem \ref{main}. The author appreciates important remarks about Conjecture \ref{conjF} by Changjian Fu. Toshiya Yurikusa gives helpful advice about Theorem \ref{sub}.
The author received generous support from Tomoki Nakanishi. The author also thanks Haruhisa Enomoto, Yoshiki Aibara, and Naohiro Tsuzu. This work was supported by JSPS KAKENHI Grant number JP20J12675.
\section{Proof of Theorem \ref{main}}
In this section, we will prove Theorem \ref{main}. We start with proving the special case. 
For any cluster pattern $v\mapsto \Sigma_v$, we fix a seed $\Sigma_s$ such that $B_s$ is bipartite. We define the \emph{source mutation} $\mu_+$ and the \emph{sink mutation} $\mu_-$ as 
\begin{align}
\mu_+=\prod_{\varepsilon(k)=1}\mu_k,\quad \mu_-=\prod_{\varepsilon(k)=-1}\mu_k,
\end{align}
where $\varepsilon$ is the sign induced by the bipartite matrix $B_s$ (see \eqref{bipartitecondition}). The \emph{bipartite belt} induced by $\Sigma_s$ consists of seeds $\Sigma_t$ satisfying the following condition: there exists a mutation sequence $\mu$ consisting of $\mu_+$ and $\mu_-$ such that $\Sigma_t=\mu(\Sigma_s)$.

\begin{remark}
Definition of a bipartite belt in this paper is a generalised version of \cite{fziv}*{Definition 8.2}. We do not assume that the initial exchange matrix $B$ is bipartite. A bipartite belt in \cite{fziv} corresponds with that induced by the initial bipartite seed $\Sigma_{t_0}$ in this paper.
\end{remark}
\begin{lemma}[{\cite{fziv}*{Corollary 10.10}}]\label{bipartitebeltcase}
In any cluster algebra, if the initial matrix $B$ is bipartite and $\Sigma_t$ belongs to the bipartite belt induced by $\Sigma_{t_0}$, then we have \eqref{fdeq}.
\end{lemma}
By Remark \ref{cartanexchange}, if $\mathcal{A}$ is of finite type, then $\mathcal{A}$ has a seed whose exchange matrix is bipartite. We prove the case that the initial matrix $B$ is bipartite.
\begin{lemma}[\cite{fziv}*{Proposition 11.1 (1)}]\label{finiteclusterspecial}
In a cluster algebra of finite type, for a bipartite seed $\Sigma_s$, every cluster variable belongs to a seed lying on the bipartite belt induced by $\Sigma_s$.
\end{lemma}

\begin{proposition}\label{bipartitef=d}
We fix a cluster algebra of finite type whose initial matrix $B$ is bipartite. For any $i\in\{1,\dots,n\}$ and $t\in\TT_n$, we have \eqref{fdeq}.
\end{proposition}
\begin{proof}
It follows from Lemmas \ref{bipartitebeltcase}
and \ref{finiteclusterspecial}.
\end{proof}
Let us generalize Proposition \ref{bipartitef=d} to the case that the initial matrix $B$ is non-bipartite. 
The next lemma is a generalization of Lemma \ref{finiteclusterspecial}. 
\begin{lemma}\label{finitecluster}
In a cluster algebra of finite type, for a seed $\Sigma_s$, every cluster variable belongs to seeds lying on the bipartite belt induced by $\Sigma_s$.
\end{lemma}
\begin{proof}
Let $\Sigma^{B}_t$ be a seed and $\Sigma^{B'}_s$ a bipartite seed. By regarding a change of the initial seed from $\Sigma^{B}_t$ to $\Sigma^{B'}_s$ as a change from the expression of cluster variables and coefficients by $\Sigma^B_t$ to that by $\Sigma^{B'}_s$, the general cases follows from the bipartite cases.
\end{proof}
We introduce a key lemma.
\begin{lemma}[{\cite{rs}*{Theorem 2.2}, \cite{fg}*{Theorem 3.10}}]\label{FDdual}
\begin{itemize}
\item[(1)]In a cluster algebra $\Acal(B)$ of finite type, for $t\in\TT_n$, we have
\begin{align}
D_t^{B;t_0}=(D_{t_0}^{B_t^T;t})^T.\label{ddual}
\end{align}
\item[(2)]In any cluster algebra $\Acal(B)$, for $t\in\TT_n$, we have
\begin{align}
F_t^{B;t_0}=(F_{t_0}^{B_t^T;t})^T.
\end{align}
\end{itemize}
\end{lemma}
\begin{remark}
In \cite{rs}*{Theorem 2.2}, the duality for $D$-matrices is given by
\begin{align}\label{ddual'}
D_t^{B;t_0}=(D_{t_0}^{-B_t^T;t})^T.
\end{align}
The equation \eqref{ddual} derives from \eqref{ddual'}. In fact, by symmetry of the recursion \eqref{dvectorrecursion} of $d$-vectors, we have $D_{t_0}^{-B_t^T;t}=D_{t_0}^{B_t^T;t}$.
\end{remark}
We are ready to prove the main theorem in this paper.
\begin{proof}[Proof of Theorem \ref{main}] 
We fix a bipartite seed $\Sigma$ in $\Acal(B)$. Note that $\mathcal{A}(B)$ is of finite type if and only if $\mathcal{A}(B^T)$ is also. Moreover, $B^T_t$ is bipartite if and only if $B_t$ is bipartite.
Therefore, $\mathcal{A}(B^T)$ is of finite type, and for any $t$ in a bipartite belt induced by $\Sigma$, $B_t^T$ is bipartite. Thus, we have 
\begin{align}\label{FDtranspose}
F_{t_0}^{B_t^T;t}=\left[D_{t_0}^{B_t^T;t}\right]_+,
\end{align}
by Proposition \ref{bipartitef=d} (the operation $[\ ]_+$ on matrices are performed component-wise).
Therefore, we have 
\begin{align}\label{FD}
F_{t}^{B;t_0}=\left[D_t^{B;t_0}\right]_+,
\end{align}
by Proposition \ref{FDdual}.
By Lemma \ref{finitecluster}, for a cluster variable $x_{j;s}$, there exist $i\in\{1,\dots,n\}$ and a vertex $t$ of the bipartite belt induced by a seed $\Sigma$ such that $x_{j;s}=x_{i;t}$. Thus, $\ff_{j;s}=\ff_{i;t}=\dd_{i;t}=\dd_{j;s}$ by \eqref{FD}, and we have \eqref{fdeq} for any initial vertex $t_0$.
\end{proof}
\section{Proof of Theorem \ref{sub} (1)}
In this section, we prove Theorem \ref{sub} (1). We fix any $\Acal(B)$ of finite type. Through this section, unless otherwise noted, we assume that seeds, cluster variables, clusters, $f$-vectors, $d$-vectors, $F$-matrices, and $D$-matrices are those of $\Acal(B)$. We start with proving the special case. We say that a vector $\bb$ is \emph{positive} (resp. \emph{negative}) if $\bb\neq \mathbf 0$ and all entries of $\bb$ is non-negative (resp. non-positive). Due to Theorem \ref{main}, we can use the properties of $d$-vectors to prove Theorem \ref{sub} (1). 
\begin{lemma}[{\cite{cp}*{Corollary 3.5}}]\label{dpositive}
A cluster variable $x_{i;t}$ is not in the initial cluster if and only if $\dd_{i;t}$ is positive.
\end{lemma}
By this lemma, we have the following corollary:
\begin{corollary}\label{0-initial}
An f-vector $\ff_{i;t}$ is the zero-vector if and only if $x_{i;t}$ is in the initial cluster.
\end{corollary}
\begin{proof}
The ``if" part is clear. We prove the ``only if" part. By Theorem \ref{main}, $\ff_{i;t}=\mathbf{0}$ implies that $\dd_{i;t}$ is negative or $\mathbf{0}$. By Lemma \ref{dpositive}, $x_{i;t}$ is in the initial cluster.
\end{proof}

The following propositions and corollary are essential for proving Theorem \ref{sub}:

\begin{proposition}[{\cite{fziv}*{Theorem 11.1 (2)}}]\label{bijspcase}
We fix a cluster algebra $\Acal(B)$ of finite type such that $B$ is bipartite and Cartan finite $X_n$ type. The $d$-vectors establish a bijection between cluster variables and the set of all almost positive roots $\Phi_{\geq -1}=\Phi_+\cup -\Delta$ of $X_n$ Dynkin type, where $\Phi_+$ is the set of all positive roots and $-\Delta$ is the set of negative simple roots.
\end{proposition}
Let $\mathcal{D}(B)$ be the set of all $d$-vectors in $\Acal(B)$.
\begin{proposition}[{\cite{ns}*{Theorem 1.3.3}}]\label{cardinary}
We fix a cluster algebra $\Acal(B)$ of finite type. Then the cardinality $|\mathcal{D}(B)|$ depends only on the Dynkin type $X_n$ of $\Acal(B)$.
\end{proposition}

\begin{corollary}\label{duniquely}
If $\dd_{i;t}=\dd_{j;s}$ holds, then we have $x_{i;t}=x_{j;s}$.
\end{corollary}
\begin{proof}
Let $B'$ be a bipartite matrix of finite Cartan $X_n$ type which is mutation equivalent to $B$. Then by Proposition \ref{bijspcase} and Proposition \ref{cardinary}, we have 
\begin{align}
|\mathcal{D}(B)| = |\mathcal{D}(B')| = |\Phi_{\geq -1}|.
\end{align}
Let $\Xcal(B)$ be the set of all cluster variables of $\Acal(B$). By Remark \ref{cartanexchange} and Proposition \ref{bijspcase}, we have
\begin{align}
|\mathcal{X}(B)| = |\mathcal{X}(B')| = |\Phi_{\geq -1}|.
\end{align}
Therefore, we have
\begin{align}\label{d=x}
|\mathcal{D}(B)| = |\mathcal{X}(B)|.
\end{align}
If there exist $d$-vectors $d_{i;t}$ and $d_{j;s}$ such that $d_{i;t}=d_{j;s}$ and $x_{i;t}\neq x_{j;s}$, then we have $|\mathcal{D}(B)| < |\mathcal{X}(B)|$. This conflicts with \eqref{d=x}. 
\end{proof}
By Corollary \ref{0-initial} and Corollary \ref{duniquely}, we have the following proposition:
\begin{proposition}\label{fneq0case}
If $\ff_{i;t}=\ff_{j;s}\neq\mathbf 0$, then we have $x_{i;t}=x_{j;s}$.
\end{proposition}

\begin{proof}
Let $\ff$ be an $f$-vector which is not equal to $\mathbf{0}$. We assume that $\ff=\ff_{i;t}=\ff_{j;s}$. Since all entries of $\ff$ are non-negative, and the $f$-vector is not equal to $\mathbf{0}$, we have $\ff=\dd_{i;t}=\dd_{j;s}$ by Theorem \ref{main} and Lemma \ref{dpositive}. By Proposition \ref{duniquely}, we have $x_{i;t}=x_{j;s}$.
\end{proof}
While $d$-vectors can distinguish the initial clusters, $f$-vectors cannot. Thus, we cannot detect the initial cluster variables contained in a cluster by their $f$-vectors directly. However, using the property of $d$-vectors, we can detect them.

\begin{proposition}\label{ddetect}
For a $D$-matrix $D_t^{B;t_0}$, negative column vectors of $D_t^{B;t_0}$ are uniquely determined by positive column vectors of $D_t^{B;t_0}$.
\end{proposition}
\begin{proof}
By \eqref{ddual}, the transposition of a $D$-matrix in a cluster algebra of finite type is another $D$-matrix in a cluster algebra of finite type because $\Acal(B)$ is of finite type if and only if $\Acal(B_t^T)$ is of finite type. Since negative $d$-vectors have the form of $-\ee_i$, if the $(i,j)$ entry of $D_t^{B;t_0}$ is $-1$, then entries of the $i$th row and the $j$th column of $D_t^{B;t_0}$ are all 0 except for the $(i,j)$-entry. Since $D_t^{B;t_0}$ do not have the zero column vector by Lemma \ref{dpositive}, if a $D$-matrix has just $m$ positive columns, then we have just $n-m$ indices $i_1,\dots,i_{n-m}$ such that the $i_k (k\in\{1,\dots,n-m\})$th entry of all positive $d$-vectors is $0$, and $D_t^{B;t_0}$ has column vectors $-\ee_{i_k} (k\in\{1,\dots,n-m\})$. This finishes the proof.
\end{proof}
We are ready to prove Theorem \ref{sub} (1).
\begin{proof}[Proof of Theorem \ref{sub} (1)]
If $\ff_{i;t}=\ff_{j;s}\neq\mathbf{0}$, then we have $x_{i;t}=x_{j;s}$ by Proposition \ref{fneq0case}. We assume that there are $m$ zero-vectors in $(\ff_{1;t},\dots,\ff_{n;t})$ (or $(\ff_{1;s},\dots,\ff_{n;s})$). By regarding positive $f$-vectors as $d$-vectors by Theorem \ref{main}, we detect the rest of $d$-vectors in $\xx_t$ and $\xx_s$ by Proposition \ref{ddetect}. Since positive $d$-vectors in $\xx_t$ corresponds with that of $\xx_s$, we have $\xx_t=\xx_s$ by Corollary \ref{duniquely}.
\end{proof}

\section{Proof of Theorem \ref{sub} (2)}
We prove Theorem \ref{sub} (2). The strategy of this proof is almost the same as Theorem \ref{sub} (1), but we sometimes use the special properties of cluster algebras of rank $2$. 

For a cluster algebra of rank $2$, we may assume that the initial matrix $B$ has the following form without loss of generality:
\begin{align}\label{rank2mat}
B=\begin{bmatrix}
0&b\\
-c&0
\end{bmatrix},\quad b,c\in \ZZ_{\geq0}, \quad bc\geq4,
\end{align}
 because when $bc\leq 3$, this cluster algebra is of finite type. We name vertices of $\TT_2$ by the rule of \eqref{A2tree} and consider a cluster pattern $t_n\mapsto (\xx_{t_n},\yy_{t_n}, B_{t_n})$. We abbreviate $\xx_{t_n}$ (resp., $\yy_{t_n},\ B_{t_n},\Sigma_{t_n}$) to $\xx_n$ (resp., $\yy_{n},\ B_{n},\ \Sigma_n$). We also abbreviate $d$-vectors, $D$-matrices, $f$-vectors, and $F$-matrices in the same way. 
 
We consider a description of $D$-matrices in the case $n\geq0$. First, we have
 \begin{align}
 D_{0}^{B;t_0}=\begin{bmatrix}-1&0\\0&-1\end{bmatrix},\quad D_{1}^{B;t_0}=\begin{bmatrix}1&0\\0&-1\end{bmatrix} 
\end{align}
by direct calculation. By \cite{llz}*{(1.13)}, if $n>0$ is even, then we can denote
\begin{align}
D_n^{B;t_0}=\begin{bmatrix}
S_{\frac{n-2}{2}}(u)+S_{\frac{n-4}{2}}(u)& bS_{\frac{n-2}{2}}(u)\\
cS_{\frac{n-4}{2}}(u) &S_{\frac{n-2}{2}}(u)+S_{\frac{n-4}{2}}(u)
\end{bmatrix},
\end{align}
and if $n>1$ is odd, then we can denote
\begin{align}
D_n^{B;t_0}=\begin{bmatrix}
S_{\frac{n-1}{2}}(u)+S_{\frac{n-3}{2}}(u)& bS_{\frac{n-3}{2}}(u)\\
cS_{\frac{n-3}{2}}(u) &S_{\frac{n-3}{2}}(u)+S_{\frac{n-5}{2}}(u)
\end{bmatrix},
\end{align}
where $u=bc-2$ and $S_p(u)$ is a (normalized) Chebyshev polynomial of the second kind, that is, 
\begin{align}
S_{-1}(u)=0,\quad S_0(u)=1,\quad S_p(u)=uS_{p-1}(t)-S_{p-2}(u)\ (p\in\mathbb{N}).
\end{align}

When $n<0$, $D_n^{B;t_0}$ is the following matrix:
\begin{align}\label{negativen}
D_n^{B;t_0}=\begin{bmatrix}
d_{22;-n}^{-B^T}& d_{21;-n}^{-B^T}\\
d_{12;-n}^{-B^T}& d_{11;-n}^{-B^T}
\end{bmatrix},
\end{align}
where $d_{ij;-n}^{-B^T}$ is the $(i,j)$ entry of $D_{-n}^{-B^T;t_0}$. 

We fix any $\Acal(B)$ of rank 2. Through the rest of this section, unless otherwise noted, we assume that seeds, cluster variables, clusters, $f$-vectors, $d$-vectors, $F$-matrices, and $D$-matrices are those of $\Acal(B)$.
Using the above descriptions, we prove some properties for $d$-vectors.
\begin{lemma}\label{dpositive2}
The initial cluster variables belong to $\Sigma_{0}$ or $\Sigma_{\pm1}$. Furthermore, $x_{i;t}$ is not in the initial cluster if and only if $\dd_{i;t}$ is positive.
\end{lemma}
\begin{proof}
We prove it in the case $n>0$. It suffices to show that for any $u\geq 2$ and $p\geq -1$, $S_p(u)\geq 0$ holds and $S_p(u)=0$ if and only if $p=-1$. The general term of $S_p(u)$ is 
\begin{align}
S_p(u)=\begin{cases}
p+1 &\text{if $u=2$};\\
\dfrac{1}{\sqrt{u^2-4}}\left(\left(\dfrac{u+\sqrt{u^2-4}}{2}\right)^{p+1}-\left(\dfrac{u-\sqrt{u^2-4}}{2}\right)^{p+1}\right) &\text{if $u\neq 2$}.
\end{cases}
\end{align}
By direct calculation, we have $S_p(u)\geq 0$. Also, $S_p(u)=0$ holds if and only if $p=-1$ holds. 
In the case $n<0$, we can use the result of the case $n>0$ by \eqref{negativen}.
\end{proof}

The following corollary is analogous to Corollary \ref{0-initial}:

\begin{corollary}\label{0-initial2}
An $f$-vector $\ff_{i;t}$ is the zero-vector if and only if $x_{i;t}$ is in the initial cluster.
\end{corollary}

\begin{proof}
We can prove it in the same way as Corollary \ref{0-initial}: we use Lemma \ref{dpositive2} instead of Lemma \ref{dpositive}.
\end{proof}

The following lemma is analogous to Corollary \ref{duniquely}:

\begin{lemma}\label{duniquely2}
If $\dd_{i;t}=\dd_{j;s}$, then we have $x_{i;t}=x_{j;s}$.
\end{lemma}
\begin{proof}
When $\PP=\{1\}$, by using \cite{llz}*{(1.15)} (cf. Section \ref{restoreformula}), we have the expressions of cluster variables induced by $d$-vectors. For the general case, the use of \cite{cl2}*{Proposition 3 (i)} leads to the case where $\PP=\{1\}$.
\end{proof}

The following proposition is analogous to Corollary \ref{fneq0case}:

\begin{proposition}\label{fneq0case2}
If $\ff_{i;t}=\ff_{j;s}\neq \mathbf 0$, then we have $x_{i;t}=x_{j;s}$.
\end{proposition}

\begin{proof}
We can prove it in the same way as Corollary \ref{fneq0case}: we use Corollary \ref{0-initial2} and Lemma \ref{duniquely2} instead of Corollary \ref{0-initial} and Corollary \ref{duniquely} respectively.
\end{proof}

The following proposition is analogous to Proposition \ref{ddetect}. Unlike Proposition \ref{ddetect}, we do not need to use the duality for $D$-matrices.

\begin{proposition}\label{ddetect2}
For a $D$-matrix $D_n^{B;t_0}$, negative column vectors of $D_n^{B;t_0}$ are uniquely determined by positive column vectors of $D_n^{B;t_0}$.
\end{proposition}
\begin{proof}
When both $d$-vectors in $D_n^{B;t_0}$ are negative vectors, it is clear. Therefore, we can assume that only one $d$-vector is negative. By Lemma \ref{dpositive2}, the initial cluster variables only appear in $\Sigma_{0}$ or $\Sigma_{\pm1}$. Therefore, if $\dd_{1;0}=\dd_{1;-1}=-\ee_1$ is contained in two $d$-vectors associated with a cluster, then the other is always $\dd_{2:-1}$. Similarly, if $\dd_{2;0}=\dd_{2;1}=-\ee_2$ is contained in two $d$-vectors, then the other is always $\dd_{1:1}$. By this observation, it suffices to show $\dd_{2;-1}\neq \dd_{1;1}$. We have $\dd_{2;-1}=\begin{bmatrix}
0\\
1
\end{bmatrix}$, and $\dd_{1;1}=\begin{bmatrix}
1\\
0
\end{bmatrix}$ by direct calculation. This finishes the proof.
\end{proof}
We are ready to prove Theorem \ref{sub} (2).
\begin{proof}[Proof of Theorem \ref{sub} (2)]
We can prove it in the same way as Theorem \ref{sub} (1): we use Lemma \ref{duniquely2}, Proposition \ref{fneq0case2}, and Proposition \ref{ddetect2} instead of Corollary \ref{duniquely}, Proposition \ref{fneq0case}, and Proposition \ref{ddetect} respectively. 
\end{proof}
\section{Restoration formula of cluster algebras of rank 2}\label{restoreformula}
We proved that cluster variables are uniquely determined by their $f$-vectors for cluster algebras of rank 2 in the previous section. In this section, we describe these cluster variables explicitly in the case that coefficients are the principal ones. By this description, we establish a way to restore $F$-polynomials from $f$-vectors. Throughout this section, we assume that $\Acal(B)$ has the following initial matrix:

\begin{align}\label{rank2mat2}
B=\begin{bmatrix}
0&b\\
-c&0
\end{bmatrix},\quad b,c\in \ZZ_{\geq0}.
\end{align}
We do not assume $bc\geq 4$, thus cluster algebras of finite type and rank 2 ($A_2,B_2,G_2$ Dynkin types) are contained.
Unless otherwise noted, we assume that seeds, cluster variables, clusters, $f$-vectors, $d$-vectors, $F$-matrices, and $D$-matrices are those of $\Acal(B)$.

A previous work \cite{llz} has given a cluster expansions formula in the case that $\PP=\{1\}$. This formula restores the expressions of cluster variables by the initial ones from their $d$-vectors. We start with an explanation of this formula.

We define Dyck Paths and some notations along \cite{llz}*{Section 1}. Let $(a_1, a_2)$ be a pair of non-negative integers. A \emph{Dyck path} of type $a_1\times a_2$ is a lattice path from $(0,0)$ to $(a_1, a_2)$ and it does not go above the diagonal combining $(0,0)$ with $(a_1, a_2)$. For the Dyck paths of $a_1\times a_2$ type, there is the \emph{maximal} one $\mathcal{D}^{a_1\times a_2}$. 
It is defined by the following property: for any lattice point $A$ on $\mathcal{D}$, there is no lattice points between $A$ and the crosspoint of a vertical line including $A$ and the diagonal combining $(0,0)$ with $(a_1, a_2)$. 

For $\mathcal{D}=\mathcal{D}^{a_1\times a_2}$, let $\mathcal{D}_1=\{u_1,\dots,u_{a_1}\}$ be the set of horizontal edges of $\mathcal{D}$ indexed from left to right, and $\mathcal{D}_2=\{v_1,\dots, v_{a_2}\}$ be the set of vertical edges of $\mathcal{D}$ indexed from bottom to top. 

For any $A$ and $B$ on $\mathcal{D}$, let $AB$ be the subpath of $\mathcal{D}$ starting from $A$ and going in the upper right direction along $\mathcal{D}$ until it reaches $B$. If we reach $(a_1,a_2)$ before reaching B, we restart from $(0,0)$. If $A$ and $B$ are the same lattice point, then $AA$ is the subpath which starts from $A$, then passes $(a_1,a_2)$ and ends at $A$. Here $(0,0)$ and $(a_1,a_2)$ are regarded as the same point, thus if $A=(a_1,a_2)$, then $AA$ corresponds with the maximal Dyck path. We denote by $(AB)_1$ the set of horizontal edges in $AB$, and by $(AB)_2$ the set of vertical edges in $AB$. 
Let $AB^\circ$ be the set of lattice points on the subpath $AB$ except for the endpoints $A$ and $B$.

\begin{example}
We fix $(a_1,a_2)=(5,3)$, and
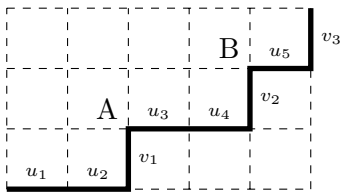
\begin{figure}[ht]
\begin{center}
$\begin{tikzpicture}
\draw[dashed, step=1,color=black] (0,0) grid (5,3);
\draw[line width=2,color=black] (0,0)--(2,0)--(2,1)--(3,1)--(4,1)--(4,2)--(5,2)--(5,3);
\draw (0.5,0) node[anchor=south] {\tiny$u_1$};
\draw (1.5,0) node[anchor=south] {\tiny$u_2$};
\draw (2.5,1) node[anchor=south] {\tiny$u_3$};
\draw (3.5,1) node[anchor=south] {\tiny$u_4$};
\draw (4.5,2) node[anchor=south] {\tiny$u_5$};
\draw (5,2.5) node[anchor=west] {\tiny$v_3$};
\draw (4,1.5) node[anchor=west] {\tiny$v_2$};
\draw (2,.5) node[anchor=west] {\tiny$v_1$};
\draw (2,1) node[anchor=south east] {A};
\draw (4,2) node[anchor=south east] {B};
\end{tikzpicture}$
\caption{A maximal Dyck path ($(a_1,a_2)=(5,3)$).}
\label{fig:Dyck-path}
\end{center}
\end{figure}
let $A=(2,1)$, $B=(4,2)$. Then 
\begin{align*}
(AB)_1=\{u_3,u_4\}, \,\, (AB)_2=\{v_2\}, 
(BA)_1=\{u_5,u_1,u_2\}, \, \, (BA)_2=\{v_3,v_1\}, 
\end{align*}

and the subpath $AA$ has length 8 (see Figure \ref{fig:Dyck-path}). 

\end{example}
Next, we define the compatibility in $\mathcal{D}$: 
\begin{definition}[{\cite{llz}*{Definition 1.10}}]
\label{df:compatible}
For $S_1\subseteq \mathcal{D}_1$, $S_2\subseteq \mathcal{D}_2$, we say that the pair $(S_1,S_2)$ is \emph{compatible} if for every $u\in S_1$ and $v\in S_2$, denoting by $E$ the left endpoint of $u$ and $F$ the upper endpoint of $v$, there exists a lattice point $A\in EF^\circ$ such that 
\begin{equation}
\label{0407df:comp}
|(AF)_1|=b|(AF)_2\cap S_2| \text{ or }|(EA)_2|=c|(EA)_1\cap S_1|.\end{equation}
\end{definition}

We are ready to describe a cluster expansion formula for cluster algebras of rank 2. 

\begin{theorem}[{\cite{llz}*{Theorem 1.11}}]
\label{th:greedy-combinatorial}
 For every $d$-vector $\dd=\begin{bmatrix}d_1 \\ d_2\end{bmatrix}$, the cluster variable $x_\dd$ corresponding to $\dd$ is given by the following equation:
\begin{align}
\label{eq:greedy-Dyck-expression}
x_{\dd} = x_1^{-d_1}x_2^{-d_2}\sum_{(S_1,S_2)}x_1^{b|S_2|}x_2^{c|S_1|},
\end{align}
where the sum is over all compatible pairs $(S_1,S_2)$ in $\mathcal{D}^{[d_1]_+\times [d_2]_+}$.
\end{theorem}
\begin{remark}
In \cite{llz}*{Theorem 1.11}, \eqref{eq:greedy-Dyck-expression} is defined for any $(a_1,a_2)\in\ZZ^2$ and is called a \emph{greedy element}.
\end{remark}
We generalize this formula to the principal coefficients version in a way which is analogous to \cite{ls}. First, we define the $g$-vectors according to \cite{fziv}. Cluster variables with the principal coefficients are homogeneous by the following $\ZZ^n$-grading: for any $i\in\{1,\dots,n\}$,
\begin{align}\label{grading}
\deg(x_i) = \ee_i, \quad \deg(y_i)=-\mathbf{b}_i,
\end{align}
where $\mathbf{b}_i$ is the $i$th column vector of $B$ (see \cite{fziv}*{Proposition 6.1}). We define the $g$-vector $\gg_{i;t}=\begin{bmatrix}
g_{1i;t}\\ \vdots\\ g_{ni;t}
\end{bmatrix}$ as the degree vector of a cluster variable $x_{i;t}$. Like $f$-vectors, they are independent of the choice of $\PP$ by defining them in the following way: for any $i\in\{1,\dots,n\}$,
\begin{align}
\gg_{i;t_0}=\ee_i,
\end{align}
and for any \begin{xy}(0,1)*+{t}="A",(10,1)*+{t'}="B",\ar@{-}^k"A";"B" \end{xy}, 
\begin{align}
g_{ij;t'}=
\begin{cases}g_{ij;t} &\text{if $j\neq k$}; \\ 
-g_{ik;t}+\mathop{\sum}\limits_{\ell=1}^n g_{i \ell;t}[b_{\ell k;t}]_+-\mathop{\sum}\limits_{\ell=1}^n b_{i \ell}[c_{\ell k;t}]_+&\text{if $j=k$},
\end{cases}\label{eq3}
\end{align}
where $c_{\ell k;t}$ is the $\ell$th entry of $\cc_{k;t}$ (cf. Remark \ref{independentf}).

When a cluster algebra is of rank 2, $g$-vectors are obtained by $d$-vectors:
\begin{theorem}\label{gdrelation}
For a $g$-vector $\gg=\begin{bmatrix}
g_1\\g_2
\end{bmatrix}$ and a $d$-vector $\dd=\begin{bmatrix}
d_1\\d_2
\end{bmatrix}$ of a cluster variable, we have the following equation: 
\begin{align}
\begin{bmatrix}
g_1\\g_2
\end{bmatrix}=\begin{bmatrix}
-d_1\\cd_1-d_2
\end{bmatrix}.
\end{align}
\end{theorem}
\begin{proof}
This is the spacial case of \cite{fziv}*{Theorem 10.12}.
\end{proof}

Using $g$-vectors, we have the following generalization of Theorem \ref{th:greedy-combinatorial}:
\begin{theorem}\label{principalformula}
For a $d$-vector $\dd=\begin{bmatrix}d_1 \\ d_2\end{bmatrix}$, the cluster variable $x_\dd$ with the principal coefficients corresponding to $\dd$ is given by the following equation:
\begin{align}\label{principal generalization}
x_{\dd} = x_1^{-d_1}x_2^{-d_2}\sum_{(S_1,S_2)}y_1^{[d_1]_+-|S_1|}y_2^{|S_2|}x_1^{b|S_2|}x_2^{c|S_1|},
\end{align}
where the sum is over all compatible pairs $(S_1,S_2)$ in $\mathcal{D}^{[d_1]_+\times [d_2]_+}$.
\end{theorem}
\begin{proof}
When a $d$-vector is the negative, we have \eqref{principal generalization} by direct calculation. We assume that a $d$-vector is positive. For any compatible pair $(S_1,S_2)\in \mathcal{D}^{d_1\times d_2}$, let $a_1(S_1,S_2)$ and $a_2(S_1,S_2)$ be integers satisfying
\begin{align}
x_{\dd} = x_1^{-d_1}x_2^{-d_2}\sum_{(S_1,S_2)}y_1^{a_1(S_1,S_2)}y_2^{a_2(S_1,S_2)}x_1^{b|S_2|}x_2^{c|S_1|}.
\end{align}
Since $x_{\dd}$ is homogeneous by the grading \eqref{grading}, and its degree is $\gg=\begin{bmatrix}
g_1\\g_2
\end{bmatrix}=\begin{bmatrix}
-d_1\\cd_1-d_2
\end{bmatrix}$ by Theorem \ref{gdrelation}, the following equation holds for any compatible pair $(S_1,S_2)$:
\begin{align}
\begin{bmatrix}
-d_1\\cd_1-d_2
\end{bmatrix}=-\begin{bmatrix}
d_1\\d_2
\end{bmatrix}+a_1(S_1,S_2)\begin{bmatrix}
0\\c
\end{bmatrix}
+
a_2(S_1,S_2)\begin{bmatrix}
-b\\0
\end{bmatrix}
+
\begin{bmatrix}
b|S_2|\\c|S_1|
\end{bmatrix}.
\end{align}
By solving the equation, we have
\begin{align}
a_1(S_1,S_2)=d_1-|S_1|,\quad a_2(S_1,S_2)=|S_2|.
\end{align}
\end{proof}
By Theorem \ref{principalformula}, definition of the $F$-polynomials, and Remark \ref{f=drank2}, we have the following restoration formula of $F$-polynomials from $f$-vectors:
\begin{corollary}
For a $f$-vector $\ff=\begin{bmatrix}
f_1\\f_2
\end{bmatrix}$, the $F$-polynomial $F_{\ff}(\yy)$ whose maximal degree vector is $\ff$ is given by the following formula:
\begin{align}
F_\ff(y_1,y_2)=\sum_{(S_1,S_2)}y_1^{f_1-|S_1|}y_2^{|S_2|},
\end{align}
where the sum is over all compatible pairs $(S_1,S_2)$ in $\mathcal{D}^{f_1\times f_2}$.
\end{corollary}
\begin{example}
Let $B=\begin{bmatrix}
0&4\\ -1&0
\end{bmatrix}$ and $\dd=\ff=\begin{bmatrix}
3\\2
\end{bmatrix}$. If $(S_1,S_2)\in\mathcal{D}^{3\times 2}$ is compatible, then at least one of the sets $S_1$ and $S_2$ is empty, or $(S_1,S_2)$ is one of pairs in the following list: 
\begin{align}
(\{u_1\},\{v_2\}), (\{u_2\},\{v_2\}),(\{u_3\},\{v_1\}).
\end{align}
Then we have an expression of the cluster variable $x_\dd$ corresponding to $d$-vector $\dd$ in $\Acal_\bullet(B)$ as follows: 
\begin{align}
x_{\dd}=\dfrac{x_1^8y_1^3y_2^2+2x_1^4y_1^3y_2+y_1^
3+3x_1^4x_2y_1^2y_2+3x_2y_1^2+3x_2^2y_1+x_2^3}{x_1^3x_2^2}.
\end{align}
Also we have the $F$-polynomial $F_{\ff}(\yy)$ corresponding to the $f$-vector $\ff$ as follows:
\begin{align}
F_{\ff}(\yy)=y_1^3y_2^2+2y_1^3y_2+y_1^
3+3y_1^2y_2+3y_1^2+3y_1+1.
\end{align}
\end{example}
\bibliography{myrefs}
\end{document}